\documentclass{amsart}

\usepackage{amssymb}
\usepackage{amsthm}
\usepackage{amsmath}
\usepackage{graphicx}
\usepackage{pifont}
\usepackage{txfonts}

\theoremstyle{plain}
\newtheorem{theorem}{Theorem}[section]

\newtheorem{proposition}[theorem]{Proposition}

\newtheorem*{MainTheorem*}{Main Theorem}
\theoremstyle{definition}
\newtheorem{definition}[theorem]{Definition}

\theoremstyle{remark}
\newtheorem{remark}[theorem]{Remark}

\newcommand{\rank}{\mathop{\mathrm{rank}}\nolimits}

\newcommand{\Z}{{\mathbb Z}}

\newcommand{\C}{{\mathbb C}}
\newcommand{\F}{{\mathbb F}}

\newcommand{\SL}[1][2]{{\mathrm{SL}_{#1}(\C)}}
\newcommand{\GL}{\mathrm{GL}}

\newcommand{\trace}{{\rm tr}\,}
\newcommand{\I}{I}
\newcommand{\bm}[1]{\mbox{\boldmath{$#1$}}}

\newcommand{\bnd}[1]{\partial_{#1}}
\newcommand{\basisM}[1][i]{\mathbf{c}^{#1}}
\newcommand{\basisBt}[1][i]{\tilde{\mathbf{b}}^{#1}}

\newcommand{\knotexterior}{E_K}

\newcommand{\univcover}[1]{{\widetilde #1}}
\newcommand{\lift}[1]{\tilde{#1}}

\newcommand{\knotgroup}{\pi_1(\knotexterior)}
\newcommand{\Reps}{R(\knotexterior)}
\newcommand{\Conjs}{\hat{R}(\knotexterior)}
\newcommand{\Comp}[2]{\hat{R}_{#1, #2}(\knotexterior)}
\newcommand{\Sym}[1]{\mathop{\mathrm{Sym}^{#1}(\C^2)}}

\newcommand{\Tor}[2]{\mathop{\mathrm{Tor}}\nolimits (#1;#2)}
\newcommand{\TorCpx}[1]{\mathop{\mathrm{Tor}}\nolimits (#1)}

\newcommand{\ie}{i.e.,\,}
\newcommand{\vol}[1]{\mathop{\mathrm{Vol}}\nolimits (#1)}
\begin{document}


\title[Higher dimensional R-torsion for torus knot exteriors]{
  Higher even dimensional Reidemeister torsion \\
  for torus knot exteriors
}

\author{Yoshikazu Yamaguchi}

\address{Department of Mathematics,
  Tokyo Institute of Technology
  2-12-1 Ookayama, Meguro-ku Tokyo, 152-8551, Japan}
\email{shouji@math.titech.ac.jp}


\keywords{torus knots; irreducible representations; Reidemeister torsion; hyperbolic volume}
\subjclass[2010]{Primary: 57M27, 57M05, Secondary: 57M25}

\begin{abstract}
We study the asymptotics of the higher dimensional Reidemeister torsion 
for torus knot exteriors, which is related to the results 
by W.~M\"uller and P.~Menal-Ferrer \& J.~Porti 
on the asymptotics of the Reidemeister torsion and the hyperbolic volumes 
for hyperbolic $3$-manifolds.
We show that the sequence of $\frac{1}{(2N)^2} \log |\Tor{E_K}{\rho_{2N}}|$
converges to zero when $N$ goes to infinity where 
$\Tor{E_K}{\rho_{2N}}$ is the higher dimensional Reidemeister torsion 
of a torus knot exterior and an acyclic $\SL[2N]$-representation of 
the torus knot group. 
We also give a classification for $\SL$-representations of 
torus knot groups, which induce acyclic $\SL[2N]$-representations.
\end{abstract}


\maketitle

\section{Introduction}
For an oriented hyperbolic $3$-manifold of finite hyperbolic volume,
the hyperbolic volume is related to the asymptotic behavior for
the sequence of the Reidemeister torsions
with respect to $\SL[n]$-representations of the fundamental group.
Here we are concerned with $\SL[n]$-representations
defined as the composition of an $\SL$-representation corresponding to 
the complete hyperbolic structure and 
the $n$-dimensional irreducible representation of $\SL$.
For the sequence $\{ \rho_n \}$ given by $\SL[n]$-representations $\rho_n$,
we can define the sequence of the Reidemeister torsions corresponding to $\rho_n$.
We are interested in finding the asymptotic behavior for 
the sequence of the Reidemeister torsions on $n$.
The above relation was shown by W.~M\"uller~\cite{Muller:AsymptoticsAnalyticTorsion}
in the context of the Ray--Singer analytic torsion for closed hyperbolic $3$-manifolds
and then it  was extended to the cusped hyperbolic $3$-manifolds,
which is according to P.~Menal-Ferrer and J.~Porti~\cite{FerrerPorti:HigherDimReidemeister}.

Menal-Ferrer and Porti take advantage of the combinatorial description of 
the Reidemeister torsion to establish the relation between 
the asymptotic behavior for a sequence of
the Reidemeister torsions and the hyperbolic volume of a cusped hyperbolic $3$-manifold $M$:
$$
\lim_{n \to \infty}
\frac{
  \log |\Tor{M}{\rho_n}|
}{
  n^2
}
=
-\frac{
  \vol{M}
}{
  4\pi
}
$$
where $\Tor{M}{\rho_n}$ denotes the Reidemeister torsion
for the $\SL[n]$-representation $\rho_n$ of $\pi_1(M)$.

We can also take advantage of the combinatorial feature of the Reidemeister torsion
to consider the case for non--hyperbolic manifolds $M$
and expect that the sequence, given by 
$
\frac{1}{n^2}  \log |\Tor{M}{\rho_n}|,
$
converges to zero when $n$ goes to $\infty$.
In this paper, we will show that this expected asymptotic behavior holds for all torus knots exteriors
and even dimensional {\it acyclic} representations $\rho_{2N}$
(see the subsequent of Definition~\ref{def:twistedcomplex} for the terminology ``acyclic'').
Here we restrict our attention to even dimensional acyclic representations of torus knot groups.
This is related to Raghunathan's cohomology vanishing Theorem~\cite{Raghunathan65:CohomologyVanishing} and the result
of~\cite{FerrerPorti:TwistedCohomology} by Menal-Ferrer and Porti.
Those results show that 
all even dimensional $\rho_{2N}$ given by a complete hyperbolic structure are acyclic.
Since we have no $\SL$-representations corresponding to hyperbolic structures for 
non--hyperbolic $3$-manifolds, we need to consider all similar representations 
of torus knot groups to discuss the asymptotics of the Reidemeister torsions.
For the sequence of the acyclic $\SL[2N]$-representations $\rho_{2N}$,
we will establish the following asymptotic behavior of 
the Reidemeister torsions for a torus knot exterior.
\begin{MainTheorem*}[Theorem~\ref{thm:main}]
  Let $\knotexterior$ denote $S^3 \setminus N(K)$ where $K$ is a torus knot
  and $N(K)$ is an open tubular neighbourhood of $K$.
  Suppose that all induced $\SL[2N]$-representations $\rho_{2N}$ by
  an irreducible $\SL$-representation $\rho$
  are acyclic.
  For the torus knot exterior $\knotexterior$ and the sequence $\{\rho_{2N}\}_{N=1, 2, \ldots}$, it holds that
  $$
  \lim_{N \to \infty}
  \frac{
    \log |\Tor{E_K}{\rho_{2N}}|
  }{
    (2N)^2
  }
  =0.
  $$
\end{MainTheorem*}
Our theorem is derived from explicit computations of the Reidemeister torsions
for torus knot exteriors in Proposition~\ref{prop:Rtorsion_torusknots}
and the existence of upper and lower bounds 
depending only on $2N$.  To complete our observation, we have to
find which irreducible $\SL$-representations of a torus knot group
induce acyclic $\SL[2N]$-representations. 
In Proposition~\ref{prop:acyclicity_twistedcomplex}, we also give
a necessary and sufficient condition for all $\rho_{2N}$ to be acyclic.
This is given by the condition that 
an irreducible $\SL$-representation $\rho$ sends a generator of the center in a torus knot group
to the non--trivial central element $-\I$ in $\SL$.

It is known that every torus knot exterior admits the structure of a Seifert fibration.
Our observation in this paper is related to 
very simple Seifert fibered pieces in JSJ decomposition of $3$-manifolds.
It is natural to expect the same convergence on the higher dimensional Reidemeister torsion
for more general Seifert fibrations but we will not develop this point here.
For example, it is possible to find the convergence on the leading coefficient 
and determine the limit for a closed Seifert manifold.
We refer to~\cite{Yamaguchi:HigherDimRtorsionSeifert} for closed Seifert manifolds.
Our main theorem provides a first step in developing a theory of 
the higher dimensional Reidemeister torsion for general 
geometric pieces in $3$-manifolds.

\section{Preliminaries}
\subsection{Reidemeister torsion}
\label{section:Rtorsion}
\subsubsection*{Torsion for acyclic chain complexes}
{\it Torsion} is an invariant which is defined for a based chain complex.
We denote by $C_*$ a {\it based} chain complex:
$$
C_*: 0 \to C_n \xrightarrow{\bnd{n}}
C_{n-1} \xrightarrow{\bnd{n-1}}
\cdots \xrightarrow{\bnd{2}}
C_1 \xrightarrow{\bnd{1}}
C_0 \to 0
$$
where each chain module $C_i$ is a vector space over a field $\F$
and equipped with a basis $\basisM$. 
We are mainly interested in based chain complexes $C_*$ whose 
homology groups vanish, \ie
$H_*(C_*) = \bm{0}$.
Such chain complexes are referred to as being {\it acyclic}.
A chain complex $C_*$ also has a basis determined by the 
boundary operators $\bnd{i}$,
which arises from the following decomposition of chain modules.
Roughly speaking, the torsion provides a property of a based chain complex
in the difference between the given basis and a new one determined
by the boundary operators.

We suppose that a based chain complex $(C_*, \basisM)$ is acyclic.
For each boundary operator $\bnd{i}$,
let $Z_i \subset C_i$ denote the kernel and
$B_{i-1} \subset C_{i-1}$ the image by $\bnd{i}$.
The chain module $C_i$ is expressed as the direct sum of $Z_i$ and the lift of $B_{i-1}$,
denoted by $\lift{B}_{i}$.
Moreover we can rewrite the kernel $Z_i$ as
the image of boundary operator $\bnd{i+1}$:
\begin{align*}
  C_i
  &= Z_i \oplus \lift{B}_{i} \\
  &= \bnd{i+1}\lift{B}_{i+1} \oplus \lift{B}_{i} 
\end{align*}
where $Z_i = B_{i}$ is written as $\bnd{i+1}\lift{B}_{i+1}$.

We denote by $\basisBt$ a basis of $\lift{B}_{i}$.
Then the set $\bnd{i+1}(\basisBt[i+1]) \cup \basisBt$
forms a new basis of the vector space $C_i$.
We define the {\it torsion} of $C_*$ as
the following alternating product of determinants of the base change matrices:
\begin{equation}
  \label{eqn:def_torsion_complex}
  \mathrm{Tor}(C_*, \basisM[*])
  = \prod_{i \geqq 0} 
  \left[
    \bnd{i+1}(\basisBt[i+1]) \cup \basisBt \,/\, \basisM
    \right]^{(-1)^{i+1}}
  \in \F^* = \F \setminus \{0\}
\end{equation}
where $[ \bnd{i+1}(\basisBt[i+1]) \cup \basisBt \,/\, \basisM ]$
denotes the determinant of the base change matrix from
the given basis $\basisM$ to 
the new one $\bnd{i+1}(\basisBt[i+1]) \cup \basisBt$.

Note that the right hand side is independent of the choice of bases $\basisBt$.
The alternating product in~\eqref{eqn:def_torsion_complex} is determined by the based chain complex
$(C_*, \basisM[*])$.

\subsubsection*{Reidemeister torsion for CW--complexes}
We will consider the torsion of the {\it twisted chain complex} given by 
a CW--complex and a representation of its fundamental group in this paper. 
Let $W$ denote a finite CW--complex and $(V, \rho)$ denote a representation of $\pi_1(W)$,
which means the pair of a vector space $V$ over $\F$ and a homomorphism $\rho$ from $\pi_1(W)$ into $\GL(V)$.
This pair is referred to as a $\GL(V)$-representation.
\begin{definition}
  \label{def:twistedcomplex}
  We define the twisted chain complex $C_*(W; V_\rho)$ which consists of the twisted chain module given by
  $$
  C_i (W; V_\rho) :=V \otimes_{\Z[\pi_1(W)]} C_i (\univcover{W};\Z)
  $$
  where $\univcover{W}$ is the universal cover of $W$ and $C_i(\univcover{W};\Z)$
  is a left $\Z[\pi_1(W)]$-module given by the covering transformation of $\pi_1(W)$.
  In taking the tensor product, we regard $V$ as a right $\Z[\pi_1]$-module under the homomorphism $\rho^{-1}$.
  We identify a chain $\bm{v} \otimes \gamma c$ with $\rho(\gamma)^{-1}(\bm{v}) \otimes c$
  in $C_i (W; V_\rho)$.
\end{definition}
We call $C_*(W;V_\rho)$ the twisted chain complex with the coefficient $V_\rho$.
Choosing a basis of the vector space $V$, we give a basis of the twisted chain complex $C_*(W; V_\rho)$.
To be more precise, let $\{e^{i}_1, \ldots, e^{i}_{m_i}\}$ be the set of $i$-dimensional cells of $W$ and 
$\{\bm{v}_1, \ldots, \bm{v}_{d}\}$ a basis of $V$ where $d = \dim_{\F} V$. 
Choosing a lift $\lift{e}^{i}_j$ of each cell and taking tensor product with the basis of $V$,
we have the following basis of $C_i(W;V_\rho)$:
$$
\basisM(W;V)=
\{
\bm{v}_1 \otimes \lift{e}^{i}_1, \ldots, \bm{v}_d \otimes \lift{e}^{i}_1,
\ldots,
\bm{v}_1 \otimes \lift{e}^{i}_{m_i}, \ldots, \bm{v}_d \otimes \lift{e}^{i}_{m_i}
\}.
$$
We denote by $H_*(W;V_\rho)$
the homology group, which is called {\it the twisted homology group} and 
say that $\rho$ is {\it acyclic} if the twisted homology group vanishes.
Regarding $C_*(W; V_\rho)$ as a based chain complex,
we define the Reidemeister torsion for $W$ and an acyclic representation $(V, \rho)$
as the torsion of $C_*(W;V_\rho)$, \ie
\begin{equation}
  \label{eqn:def_RtorsionCW}
\mathrm{Tor}(W; V_\rho) = \mathrm{Tor}(C_*(W;V_\rho), \basisM[*](W;V))
\in \F^*.
\end{equation}
This torsion is determined up to a factor in $\{\pm \det(\rho(\gamma)) \,|\, \gamma \in \pi_1(W)\}$
since we have many choices of lifts $\lift{e}^{i}_j$
and orders and orientations of cells $e^{i}_j$.

Note that the definition~\eqref{eqn:def_RtorsionCW} does not depend
of choice of a basis $ \{\bm{v}_1, \ldots, \bm{v}_{d}\}$ since 
the Euler characteristic of $C_*(W;V_\rho)$ must be zero by the acyclicity.
More precisely, 
the torsion defined by another basis $\{\bm{u}_1, \ldots, \bm{u}_{d}\}$
is expressed as the product of the torsion defined by $\{\bm{v}_1, \ldots, \bm{v}_{d}\}$
and the following factor:
$$
[ \{\bm{u}_1, \ldots, \bm{u}_{d}\} / \{\bm{v}_1, \ldots, \bm{v}_{d}\} ]^{- \chi(W)}.
$$

\begin{remark}
 If $\dim_{\F} V$ is even and $\rho$ is an $\mathrm{SL}(V)$-representation, 
   then the torsion  $\mathrm{Tor}(W; V_\rho)$ has no indeterminacy.
\end{remark}
See Turaev's book~\cite{Turaev:2000} for more details on the Reidemeister torsion.

\subsection{Irreducible representations of $\SL$ and higher dimensional Reidemeister torsion}
We review irreducible representations of $\SL$ briefly.
The vector space $\C^2$ has the standard action of $\SL$.
It is known that the symmetric product $\Sym{n-1}$ and the induced action by $\SL$
gives an $n$-dimensional irreducible representation of $\SL$.
We can identify $\Sym{n-1}$ with the vector space $V_n$ of 
homogeneous polynomials on $\C^2$ with degree $n-1$, \ie
$$
V_n =
\mathrm{span}_{\C}\langle
z_1^{n-1}, z_1^{n-2}z_2, \ldots, z_1^{n-k-1} z_2^k, \ldots,z_1 z_2^{n-2}, z_2^{n-1}
\rangle.
$$
The action of $A \in \SL$ is expressed as 
\begin{equation}
  \label{eqn:action_SL2}
A \cdot p(z_1, z_2) = p (A^{-1} \left(\begin{smallmatrix} z_1 \\ z_2\end{smallmatrix} \right))
\end{equation}
where $p(z_1, z_2)$ is a homogeneous polynomial and the variables in the right hand are determined by 
the action of $A^{-1}$ on the column vector as a matrix multiplication.
We write $(V_n, \sigma_n)$ for the representation given by the action~\eqref{eqn:action_SL2} of $\SL$
where $\sigma_n$ denotes the homomorphism from $\SL$ into $\GL (V_n)$.

\begin{remark}
  It is known that  
  \begin{enumerate}
  \item
    each representation $(V_n, \sigma_n)$ turns into an irreducible $\SL[n]$-representation of $\SL$ and;
  \item
    every irreducible $n$-dimensional representation of $\SL$ is equivalent to $(V_n, \sigma_n)$.
  \end{enumerate}
\end{remark}

We also mention that $\SL$ acts on the right in the representation $(V_n, \sigma_n)$.
Let $\rho$ be a homomorphism from $\pi_1(W)$ into $\SL$
and $\rho_n$ the composition $\sigma_n \circ \rho$, which gives 
an $\SL[n]$-representation of $\pi_1(W)$.
If the twisted chain complex $C_*(W; V_n)$ is acyclic, 
we can consider $\Tor{W}{V_n}$.
We will drop the subscript $\rho$ in the coefficient for simplicity when no confusion can arise.
\begin{definition}
We define the higher dimensional Reidemeister torsion 
for $W$ and $\rho$
as 
$\Tor{W}{V_n}$
and denote by $\Tor{W}{\rho_n}$.
\end{definition} 

We also review the explicit form of $\sigma_n(A)$ for a diagonal matrix $A$,
which will be needed in our computations.
Let $A \in \SL$ be a diagonal matrix 
$\left(
\begin{smallmatrix}
  a & 0 \\
  0 & a^{-1}
\end{smallmatrix}
\right)$ and $p(z_1, z_2)$ the monomial $z_1^{n-k-1} z_2^k$.
Then the action of $A$ is expressed as
$$
A \cdot p(z_1, z_2) = p(a^{-1}z_1, az_2) = a^{2k-n+1} z_1^{n-k} z_2^k = a^{2k-n+1} p(z_1, z_2). 
$$
Hence the eigenvalues of $\sigma_n(A) \in \SL[n]$ are given by 
$a^{-n+1}, a^{-n+3}, \ldots, a^{n-1}$,
\ie 
the weight space of $\sigma_n$ is $\{-n+1, -n+3, \ldots, n-1\}$ and
the multiplicity of each weight is $1$.

\subsection{$\SL$-representations of torus knot groups}
Here and subsequently, let $K$ denote the $(p, q)$--torus knot and $E_K$ the knot exterior.
The Reidemeister torsion depends on the conjugacy classes of 
irreducible representations of a fundamental group.
We are interested in irreducible representations from a torus knot group into $\SL$.
We denote by $\Reps$ 
the set of irreducible $\SL$-representations of $\knotgroup$ and
by $\Conjs$ the set of conjugacy classes of representations in $\Reps$.
According to D.~Johnson~\cite{Johnson:unpublished}, 
we can describe $\Conjs$ from the presentation 
$\langle x, y \,|\, x^p = y^q \rangle$ 
of $\knotgroup$ as follows:
\begin{proposition}[\cite{Johnson:unpublished},
    Proposition~$3.7$ in \cite{KitanoMorifuji:TwistedAlexTorusKnots}]
  \label{prop:chractervar_torusknot}
  Let $K$ be the $(p, q)$-torus knot.
  Then $\Conjs$ consists of $(p-1)(q-1)/2$ components,
  which are determined by the following pairs $(a, b)$ of integers, denoted by $\Comp{a}{b}$:
  \begin{enumerate}
  \item $0 < a < p$, $0 < b < q$.
  \item $a \equiv b$ mod $2$.
  \item For any $[\rho] \in \Comp{a}{b}$, we have that
    $\trace \rho(x) = 2\cos\big(\frac{\pi a}{p}\big)$ and
    $\trace \rho(y) = 2\cos\big(\frac{\pi b}{q}\big)$.
    Moreover the irreducible $\SL$-representation $\rho$ sends $x^p (= y^q)$
    to $(-\I)^a$.
  \item
    $\trace \rho (\mu) \not = 2\cos (\pi (ra/p \pm sb/q ))$
    where $\mu$ denotes the meridian given by $x^{-r} y^s$ and
    $r$ and $s$ are integers satisfying that $ps-qr=1$.
  \end{enumerate}
  In particular, $\Comp{a}{b}$ is parametrized by $\trace \rho(\mu)$
    and the complex dimension is one.
\end{proposition}

Therefore, for every irreducible $\SL$-representation $\rho$, 
the $\SL$-matrices $\rho(x)$ and $\rho(y)$ are diagonalizable, \ie 
they are conjugate to the following diagonal matrices:
$$
\begin{pmatrix}
  e^{a\pi\sqrt{-1}/p} & 0 \\
  0 & e^{-a\pi\sqrt{-1}/p}
\end{pmatrix}
\quad
\text{and}
\quad
\begin{pmatrix}
  e^{b\pi\sqrt{-1}/q} & 0 \\
  0 & e^{-b\pi\sqrt{-1}/q}
\end{pmatrix}
$$
when the conjugacy class of $\rho$ is contained in $\Comp{a}{b}$.
The eigenvalues of $\rho_{2N}(x)$ are given by
$$
e^{\pm(2N-1)a\pi\sqrt{-1}/p}, e^{\pm(2N-3)a\pi\sqrt{-1}/p}, \ldots, e^{\pm a\pi\sqrt{-1}/p}
$$
and those of $\rho_{2N}(y)$ are given by 
$e^{\pm(2N-1)b\pi\sqrt{-1}/q}$, $e^{\pm(2N-3)b\pi\sqrt{-1}/q}, \ldots, e^{\pm b\pi\sqrt{-1}/q}$.

\section{Twisted homology groups of torus knot exteriors}
\label{section:twisted_homology}
We consider the acyclicity for 
the twisted chain complex of $\knotexterior$ defined by $\rho_{2N}$.
The twisted homology group is invariant under homotopy equivalence. We first 
construct a $2$-dimensional CW--complex $W$,
which has the same twisted homology group and the same Reidemeister torsion as those of
$\knotexterior$.
Then we observe the twisted chain complex $C_*(W;V_{2N})$ instead of $C_*(\knotexterior;V_{2N})$.

From the presentation 
$\langle x, y \,|\, x^p = y^q \rangle$ 
of $\knotgroup$, we can construct a $2$-dimensional CW--complex $W$ consisting 
of one $0$--cell $e^0$
and two $1$--cells $e^1_i$ ($i=1, 2$) and one $2$-cell $e^2$ as in Fig.~\ref{fig:CWcpxW}.
The closed loops $e^0 \cup e^1_1$ and $e^0 \cup e^1_2$ correspond to 
the generators $x$ and $y$.

\begin{figure}[ht]
  \begin{center}
    \includegraphics[scale=.45]{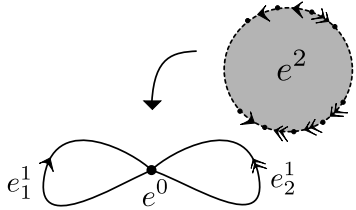}
  \end{center}
  \caption{CW--complex $W$}
  \label{fig:CWcpxW}
\end{figure}

Let $\rho$ be an irreducible $\SL$-representation 
whose conjugacy class lies in $\Comp{a}{b}$.
The composition $\rho_{2N} = \sigma_{2N} \circ \rho$ also defines the twisted chain complex
of $W$ as follows:
\begin{equation}
  \label{eqn:twistedcpx_W}
0 \to
  C_2(W;V_{2N}) \xrightarrow{\partial_2}
  C_1(W;V_{2N}) \xrightarrow{\partial_1}
  C_0(W;V_{2N}) 
\to 0.
\end{equation}
When we write $X$ and $Y$ for the matrices $\rho_{2N}(x)$ and $\rho_{2N}(y)$,
the boundary operators $\partial_i$ ($i=1, 2$) are expressed as
\begin{equation}
  \label{eqn:boundary_ops}
\partial_2
=
\begin{pmatrix}
  \I + X + \cdots X^{p-1} \\
  -(\I + Y + \cdots Y^{q-1})
\end{pmatrix}, \quad
\partial_1
=
\begin{pmatrix}
  X - \I & 
  Y - \I
\end{pmatrix}.
\end{equation}
The acyclicity of $C_*(W;V_{2N})$ is determined by the parities of the parameters $a$ and $b$ 
for the component $\Comp{a}{b}$ which contains the conjugacy class of $\rho$,
\ie the acyclicity of $C_*(\knotexterior;V_{2N})$ depends only on the component $\Comp{a}{b}$.
\begin{proposition}
  \label{prop:acyclicity_twistedcomplex}
  Suppose that  the conjugacy class of $\rho$ is contained in $\Comp{a}{b}$.
  \begin{enumerate}
  \item
    If the integers $a$ and $b$ are odd, then $C_*(\knotexterior;V_{2N})$ is acyclic for all $N \geqq 1$,
    namely $H_*(\knotexterior;V_{2N})=\bm{0}$.
  \item
    If the integers $a$ and $b$ are even, then $C_*(\knotexterior;V_{2N})$ is not acyclic for all $N \geqq 1$.
  \end{enumerate}
\end{proposition}

\begin{proof}
  It is sufficient to observe $C_*(W;V_{2N})$ instead of $C_*(\knotexterior;V_{2N})$.
  First we suppose that $a \equiv b \equiv 1$ (mod $2$).
  The equality of the matrices:
  $$
  (\I-X)(\I + X + \cdots + X^{p-1}) = 2\I.
  $$
  implies that the matrices $\I - X$ and $\I + X + \cdots + X^{p-1}$ are non--singular,
  \ie
  $\rank (\I - X) = \rank (\I + X + \cdots + X^{p-1}) = 2N$.
  By counting dimensions, we can see that the complex~\eqref{eqn:twistedcpx_W} is exact.
  
  Next we consider the case that $a \equiv b \equiv 0$ (mod $2$).
  In this case, all eigenvalues of $X$ and $Y$ are roots of unity.
  We divide this case into two parts:
  \begin{itemize}
  \item[(i)]
    either $\SL[2N]$-matrices $X$ or $Y$ does not have the eigenvalue $1$;
  \item[(ii)]
    both of $X$ and $Y$ have the eigenvalue $1$.
  \end{itemize}
  In the case (i), there is no loss of generality in assuming that $X$ does not have the eigenvalue $1$.
  Since $1+ \zeta + \cdots + \zeta^{p-1}=0$ for any $p$-th root of unity $\zeta \not = 1$,
  we have that $\I + X + \cdots + X^{p-1} = O$.
  Similarly we have that $\rank (\I + Y + \cdots + Y^{q-1}) < 2N$.
  Hence the rank of $\partial_2$ is less than $2N$,
  which implies that $H_2(W;V_{2N}) \not = \bm{0}$.

  Last we consider the case (ii).
  Without loss of generality, we can assume that
  \begin{itemize}
  \item the dimension of $V_{2N}$ is greater than $2$, \ie $N \geqq 2$;
  \item both of $p$ and $q$ are odd and;
  \item the pair $p$ and $a$ are coprime and;
  \item the pair $q$ and $b$ are also coprime.
  \end{itemize}
  The multiplicity $\xi$ of eigenvalue $1$ for $X$ is given by
  twice the number of integers $\ell$ such that $p \leqq p(2\ell-1) \leqq 2N-1$.
  Similarly we write $\eta$ for the multiplicity of eigenvalue $1$ for $Y$,
  which is given by
  twice the number of integers $m$ such that $q \leqq q(2m-1) \leqq 2N-1$.
  By the explicit form of $\partial_2$ as in~\eqref{eqn:boundary_ops},
  the kernel of $\partial_2$ coincides with the intersection
  of $\ker (\I + X + \cdots + X^{p-1})$ and $\ker (\I + Y + \cdots + Y^{q-1})$. 
  Since the ranks of $\I + X + \cdots X^{p-1}$ and $\I + Y + \cdots Y^{q-1}$
  are equal to $\xi$ and $\eta$,
  we have that
  \begin{align*}
    \dim_\C \ker \partial_2
    &= \dim_\C \ker (\I + X + \cdots + X^{p-1}) + \dim_\C \ker (\I + Y + \cdots + Y^{q-1}) \\
    & \quad - \dim_\C
    \mathrm{span}\big(
    \ker (\I + X + \cdots + X^{p-1}) \cup \ker (\I + Y + \cdots + Y^{q-1})
    \big)\\
    &\geqq \dim_\C \ker (\I + X + \cdots + X^{p-1}) + \dim_\C \ker (\I + Y + \cdots + Y^{q-1}) -2N \\
    &= 2N - \xi - \eta\\
    &\geqq 2N - \frac{2N-1}{p} - 1 - \frac{2N-1}{q} - 1.
  \end{align*}
  We proceed to find the range on $N$
  satisfying that $(2N-1)(1/p + 1/q) + 2 < 2N$.
  We can rewrite the condition $(2N-1)(1/p + 1/q) + 2 < 2N$ as 
  \begin{equation}
    \label{eqn:range_N}
    \frac{1}{1 - 1/p - 1/q} + 1 < 2N.
  \end{equation}
  Since we assume that $p \geqq 3$ and $q \geqq 5$, the inequality~\eqref{eqn:range_N}
  holds for $N \geqq 2$.
  Hence we have shown that $\dim_\C \ker \partial_2 > 0$ under our assumptions,
  which implies that $H_2(W;V_{2N})$ is nontrivial.
\end{proof}

\section{Computation of the higher dimensional Reidemeister torsion for torus knot exteriors}
The Reidemeister torsion has the invariance under simple homotopy equivalence.
In fact, a torus knot exterior $\knotexterior$ is simple homotopy equivalent to 
the $2$-dimensional CW--complex $W$ constructed as in Section~\ref{section:twisted_homology}.
Therefore we can deduce the explicit value of the Reidemeister torsion for a knot exterior $\knotexterior$
from the twisted chain complex~\eqref{eqn:twistedcpx_W} of $W$.
We assume that $\rho_{2N}$ defines the acyclic twisted chain complex $C_*(\knotexterior;V_{2N})$.
As we have seen in Section~\ref{section:twisted_homology},
such a representation $\rho_{2N}$ is given by an irreducible $\SL$-representation $\rho$ of $\knotgroup$
such that 
$\rho(x^p) = \rho(y^q) = -I$.
We first show the explicit value of the Reidemeister torsion for $\rho_{2N}$ and then
consider the asymptotic behavior on the dimension $2N$.

\subsection{Higher dimensional Reidemeister torsion for $\rho_{2N}$}
We write $\{ \bm{v}_i \,|\, i=1, \ldots, 2N \}$ for a basis of $V_{2N}$.
For the acyclic chain complex $C_*(W;V_{2N})$, we can set $\basisBt$ $(i=0, 1, 2)$ as 
\begin{align*}
  \basisBt[2] &= \{ \bm{v}_1 \otimes \lift{e}^2, \ldots, \bm{v}_{2N} \otimes \lift{e}^{2}\}, \\
  \basisBt[1] &= \{ \bm{v}_1 \otimes \lift{e}^1_2, \ldots, \bm{v}_{2N} \otimes \lift{e}^1_2\}
\end{align*}
and the others $\basisBt$ as $\emptyset$. 
The torsion of $C_*(W;V_{2N})$ is given by the alternating product:
\begin{align}
\prod_{i=0}^2 [ \bnd{i+1} \basisBt[i+1] \cup \basisBt / \basisM ]^{(-1)^{i+1}}
&=
\left[\bnd{3} \basisBt[3] \cup \basisBt[2] / \basisM[2] \right]^{(-1)} \cdot
\left[\bnd{2} \basisBt[2] \cup \basisBt[1] / \basisM[1] \right] \cdot
\left[\bnd{1} \basisBt[1] \cup \basisBt[0] / \basisM[0] \right]^{(-1)} \notag \\
&=
\frac{
  \det (\I + X + \cdot + X^{p-1})
}{
  \det (Y-\I)
} \notag \\
&=
\frac{
  2^{2N}
}{
  \det (\I -X)\det (Y-\I)
}. \label{eqn:torsion_W}
\end{align}
Here we use $(\I - X)(\I + X + \cdots X^{p-1}) = \I - X^p = 2\I$.
The determinant $\det(\I-X)$ in the denominator of Eq.~\eqref{eqn:torsion_W} is expressed as
\begin{align}
  \det(\I - X)
  &= \prod_{k=1}^{N} (1 - e^{\frac{(2k-1)a\pi\sqrt{-1}}{p}}) (1 - e^{\frac{-(2k-1)a\pi\sqrt{-1}}{p}}) \notag\\
  &= \prod_{k=1}^{N} 4 \sin^2 \left( \frac{(2k-1)a\pi}{2p}\right)
  \label{eqn:denominator_X}
\end{align}
and $\det(Y - \I)$ is also expressed  as
\begin{equation}
\label{eqn:denominator_Y}
\det(Y - \I)
  = \prod_{k=1}^{N} 4 \sin^2 \left( \frac{(2k-1)b\pi}{2q}\right).
\end{equation}
Substituting Eq.s~\eqref{eqn:denominator_X} \& \eqref{eqn:denominator_Y} into \eqref{eqn:torsion_W}, 
we obtain the explicit form of $\TorCpx{C_*(W;V_{2N})}$ which coincides with
the Reidemeister torsion of $\knotexterior$.
\begin{proposition}
  \label{prop:Rtorsion_torusknots}
  If we choose an irreducible $\SL$-representation $\rho$ whose conjugacy class lies in $\Comp{a}{b}$
  such that $a \equiv b \equiv 1$ (\textrm{mod} $2$),
  then the Reidemeister torsion for $\knotexterior$ and $\rho_{2N}$ is expressed as
  $$
  \Tor{\knotexterior}{\rho_{2N}} =
  \frac{
    2^{2N}
  }{
    \prod_{k=1}^{N}
    4^2
    \sin^2 \left( \frac{(2k-1)a\pi}{2p}\right)
    \sin^2 \left( \frac{(2k-1)b\pi}{2q}\right)
  }.
  $$
\end{proposition}

\subsection{Asymptotic behavior of the Reidemeister torsion}
We consider the asymptotic behavior of $\Tor{\knotexterior}{\rho_{2N}}$ 
in increasing the dimension of representation $\rho_{2N}$ to infinity.
Menal-Ferrer and Porti \cite{FerrerPorti:HigherDimReidemeister}
showed the asymptotic behavior of the absolute value of the Reidemeister torsion for
a hyperbolic knot exterior $M$, which is expressed as
\begin{equation}
\label{eq:entropy_limit}
\lim_{N \to \infty}\frac{\log |\Tor{M}{\rho_{2N}}|}{(2N)^2}
= -\frac{\vol{M}}{4\pi}.
\end{equation}
Every torus knot exterior admits no hyperbolic structure and 
the hyperbolic volume is zero.
Our main theorem shows that 
the equation in~\eqref{eq:entropy_limit} also holds 
for torus knot exteriors.
\begin{theorem}
  \label{thm:main}
  Under the assumption of Proposition~\ref{prop:Rtorsion_torusknots},
  we have the following limit for the absolute value of the Reidemeister torsion
  $\Tor{\knotexterior}{\rho_{2N}}$:
  \begin{equation}
    \lim_{N \to \infty}\frac{\log |\Tor{\knotexterior}{\rho_{2N}}|}{(2N)^2}
    = 0.
  \end{equation}
\end{theorem}
\begin{proof}[Proof of Theorem~\ref{thm:main}]
  This follows from Proposition~\ref{prop:Rtorsion_torusknots} and the existences on
  upper and lower bounds of $|\sin ((2k-1)a\pi/2p)|$ and $|\sin ((2k-1)b\pi/2q)|$
  for $1 \leqq k \leqq N$, which depend only on $p$ and $q$.
  Since both $a$ and $b$ are odd, it follows that
  $|\sin (\pi/2p)| \leqq |\sin ((2k-1)a\pi/2p)| \leqq 1$ and
  $|\sin (\pi/2q)| \leqq |\sin ((2k-1)b\pi/2q)| \leqq 1$.
  Hence we have the following lower and upper bounds for $|\Tor{\knotexterior}{\rho_{2N}}|$:
  $$
    \frac{
    2^{2N}
  }{
    4^{2N}
    }
    \leqq
  |\Tor{\knotexterior}{\rho_{2N}}| \leqq
  \frac{
    2^{2N}
  }{
    4^{2N}
    \sin^{2N} \left( \displaystyle{\frac{\pi}{2p}} \right)
    \sin^{2N} \left( \displaystyle{\frac{\pi}{2q}} \right)
  }.
  $$
  We can see the asymptotic behavior of $|\Tor{\knotexterior}{\rho_{2N}}|$ from
  the following inequalities:
  $$
    \frac{
      \log |\Tor{\knotexterior}{\rho_{2N}}|
    }{
      (2N)^2
    }
    \geqq
    \frac{-\log 2}{2N} 
    \quad \xrightarrow{N \to \infty} 0
    $$
    and
    $$
    \frac{
      \log |\Tor{\knotexterior}{\rho_{2N}}|
    }{
      (2N)^2
    }
    \leqq
    \frac{-\log 2}{2N}
    - \frac{
      \log \left|
      \sin \left( \displaystyle{\frac{\pi}{2p}} \right)
      \sin \left( \displaystyle{\frac{\pi}{2q}} \right)
      \right|
    }{
      2N
    }
    \quad \xrightarrow{N \to \infty} 0.
    $$
\end{proof}

\section*{Acknowledgments}
The author wishes to express his thanks to Joan Porti and Teruaki
Kitano for drawing the author's attention to the asymptotic behavior of
the higher dimensional Reidemeister torsion for torus knot exteriors.
This work was motivated in RIMS Seminar ``Representation spaces,
twisted topological invariants and geometric structures of
3-manifolds''.  This research was supported by Research Fellowships of
the Japan Society for the Promotion of Science for Young Scientists.

\bibliographystyle{amsalpha}
\bibliography{higherDimTorsionTorusKnots}
\end{document}